\documentclass[12pt]{amsart}
\usepackage[cp1250]{inputenc}
\usepackage{amsmath}
\usepackage{amsfonts}
\usepackage{amssymb}

\newtheorem{theorem}{Theorem}[section]
\newtheorem{proposition}[theorem]{Proposition}
\newtheorem{lemma}[theorem]{Lemma}
\theoremstyle{definition}
\newtheorem{example}[theorem]{Example}
\newtheorem{col}[theorem]{Colloary}
\newtheorem{remark}[theorem]{Remark}
\newtheorem{df}[theorem]{Definition}
\numberwithin{equation}{section}
\newcommand{\h}{\mathcal{H}}

  \DeclareMathOperator{\Arg}{Arg}
\newcommand{\C}{\mathbb{C}}
\author[M. Ptak]{Marek Ptak}
\address{Marek Ptak, Department of Applied Mathematics,
University of Agriculture, ul. Balicka 253c\\ 30-198 Krak\'ow, Poland.}
\email{rmptak@cyf-kr.edu.pl}

\thanks{The research of the  first author was financed by the Ministry of Science and Higher Education of the Republic of Poland}
\author[K. Simik]{Katarzyna Simik}
\address{Katarzyna Simik, Institute of Mathematics, Pedagogical University, ul. Podchor\c a\.zych 2 \\30-084 Krak\'ow, Poland.}
\email{kasia.simik@interia.pl}

\author[A. Wicher]{Anna Wicher}
\address{Anna Wicher, Institute of Mathematics, Pedagogical University, ul. Podchor\c a\.zych 2 \\30-084 Krak\'ow, Poland.}
\email{ania123wch@gmail.com}
\title[$C$--normal operators]{$C$--normal operators}


\begin{document}
\maketitle
\begin{abstract} A new class of operators, larger than $C$-symmetric operators and different than normal one, named $C$--normal operators is introduced. Basic properties are given. Characterizations of this operators in finite dimensional spaces using a relation with conjugate normal matrices are presented. 
Characterizations of Toeplitz operators and composition operators as $C$--normal operators are given. Bunches of examples are presented.
\end{abstract}

\section{Introduction and main definition}
Let $\h$ be a complex Hilbert space and denote by $L(\h)$ (by $LA(\h)$, respectively)  the algebra (the space, respectively) of all bounded linear (antilinear, respectively) operators in the space $\h$. The theory of selfadjoint and normal operators has been developed for many years. However, there are many operators which do not belong to those classes. On the other hand, a complex Hilbert space can be equipped with additional structure given by {\it conjugation} $C$, i.e. antilinear isometric involution; ($C\in LA(\h)$, $C^{2} = I $ and $\langle h, g\rangle = \langle Cg,Ch\rangle$ for all $ h,g \in \h$). Such a structure naturally appears in physics, see \cite{GPP}.
On the other hand, conjugations are related to adjoint operators in the antilinear sense. Following  Wigner, (see \cite{Uh}),
for antilinear operator $X\in LA(\h)$, there is the unique antilinear operator $X^{\sharp }$ called the {\it antilinear adjoint} of $X$ such that
\begin{equation}\label{astar}
\langle Xx, y \rangle = \langle\overline{x, X^{\sharp }y} \rangle\qquad\text{ for all}\ x, y \in \h.
\end{equation}
The antilinear operator $X$ is called antilinear selfadjoint if $ X^{\sharp}=X $. Conjugations are the examples of such operators since $C^\sharp=C$.

 Having a conjugation
 $C$ on a space $\h$, an operator $T$  can be called {\it $C$-symmetric} if $CAC=A^{*}$, see \cite{GP}. It turned out, see  \cite[Lemma 5.1]{CKP}, that operator  $A\in L(\h)$ is $C$-symmetric if and only if
$AC$ is antilineary selfadjoint, i.e. $(AC)^{\sharp}=AC$. The $C$--symmetric  operators have applications in physics especially in the quantum mechanics and the spectral analysis; let us recall monograph \cite{NM} and paper \cite{BFGJ}. Authors send the reader to \cite{GPP} for more of {\it Mathematical and physical aspects of complex symmetric operators}. It is worth to mention that $C$--symmetric operators have got interesting properties which was intensively studied, see \cite{GP,GW}. For more references see \cite{GPP}. On the other hand, many natural operators belong to this class: truncated Toeplitz, Voltera operators, normal operators and many others.



It is natural to search for the larger class of operators than $C$--symmetric ones.
 Having in mind classical selfadjoint and normal operators,
 it is natural to put forward the following
\begin{df}\label{def1} An operator  $N\in L(\h)$ is called {\it $C$-normal} if  \begin{equation}\label {dfcnor}
NC(NC)^{\sharp} = (NC)^{\sharp}(NC) .
\end{equation}
\end{df}

The definition refers to definition of normality for antilinear operators, see \cite{Uh}. Namely an antilinear operator $X\in LA(\h)$ is called {\it antylinearly normal} if
\begin{equation}\label{def3}X\,X^{\sharp } = X^{\sharp }X.\end{equation}
After stating the main definition the aim of the paper is to give equivalent conditions and basic properties of $C$--normal operator, Section 2. The next section is devoted to $C$--normal operators in finite dimensional Hilbert spaces. Section 3 shows the relation between $C$--normal operators and conjugate normal matrices; in fact we fully characterized the $C$--normal operators. The following sections concern finding a class of examples in various natural Hilbert spaces having a natural conjugations. Section 4 concerns multiplications operators in $L^2$ type spaces. Section 5 concerns Hardy space $H^2$ with some natural conjugation. Section 6 deals with composition operators. Especially interesting there are classes of $C$--normal operators being neither normal (in classical sense), nor $C$--symmetric, nor $C$--skew--symmetric. Theorems \ref{mainch2} and \ref{prop11} give  collections of such operators.
Authors think that this paper proves that $C$--normal operators form widely enough class of operators. On the other hand, we hope there will be many theorems and  properties of classical normal operators which can be moved to this new class and which should  be of the future investigations.

\section{Equivalent conditions and basic examples}
%
%
%


Let $\h$ be a complex Hilbert space with conjugation $C$. An operator $A\in L(\h)$ is called $C${\it--symmetric} if  $CAC=A^*$. It is called $C${\it--skew--symmetric} if $CAC=-A^*$.
The immediate consequence of the definition of $C$--normality (Def. \ref{def1}) is that $C$--symmetric operators and $C$--skew--symmetric operators are $C$--normal.

The paper concentrates on examples of $C$--normal operators which are neither $C$--cymmetric nor $C$--skew--symmetric, but  let us recall two classes of $C$--symmetric operators, so also $C$--normal, to give a feeling to the reader  how large and important is the class of $C$--normal operators.

\begin{example}\label{ex31}Let ${C}$ be a  conjugation in $\mathbb{C}^n$ given by $C(z_1,\dots,z_n)=(\bar{z}_n,\bar{z}_{n-1},\dots,\bar{z}_1)$. The operators are $C$--symmetric if and only if its matrix is symmetric according to ''second diagonal''. (Notations are in Sections 3,4 and this is an immediate consequence of Lemma \ref{2.1}.)
\end{example}

Let  $m$ be the normalized Lebesgue measure on the unit circle $\mathbb{T}$ and let us consider space $L^2=L^2(\mathbb{T}, m)$. The  Hardy space $H^2$ is a subspace of those elements of $L^2$ which have negative Fourier coefficient equal to $0$. One of the most interesting examples of $C$-symmetric, hence also $C$--normal, operators are truncated Toeplitz operators (TTO). (See \cite{GMR} for more details about TTO.) 
\begin{example}\label{ex2.1} By Beurling's theorem all subspaces which are invariant for the unilateral shift $S$ in the Hardy space $H^2$ ($Sf(z)=zf(z)$ for  $f\in H^2$) can be written as $\theta H^2$, where $\theta$ is an inner function.  Consider, so-called, the  {\it model space} $K^2_\theta=H^2\ominus\theta H^2$ and the orthogonal projection $P_\theta\colon L^2\to K^2_\theta$.  A {\it truncated Toeplitz operator} $A^\theta_\varphi$ with a symbol $\varphi\in L^2$ is defined as
$$A^\theta_\varphi\colon D(A^\theta_\varphi)\subset K^2_\theta\to K^2_\theta; \quad A^\theta_\varphi f=P_\theta(\varphi f)$$
for $f\in D(A^\theta_\varphi)=\{f\in K^2_\theta : \varphi f\in L^2\}$. If $A^\theta_\varphi$ is bounded, it naturally extends to the operator in $L(\h)$.  The model space $K^2_\theta$ is equipped with natural conjugation $C_\theta$, $C_\theta f= \theta \bar z \bar f$ for $f\in K^2_\theta$.  Denote by $\mathcal{T}(\theta)$ the set of all bounded truncated Toeplitz operators on $K^2_\theta$. As it was shown in \cite{[Sa],GMR} operators from $\mathcal{T}(\theta)$ are $C_\theta$--symmetric, hence $C_\theta$--normal.
\end{example}
We have the following equivalent conditions:
\begin{theorem} \label{twCnor}Let $C$ be a conjugation on $\h$ and
 let $N\in L(\h)$. The followings conditions are equivalent:
\begin{enumerate}
\item $N$ is C-normal,
\item $N^*$ is C-normal,
\item $CNC$ is C-normal,
\item $CN^*C$ is C-normal,
\item $CNN^{*} = N^{*}NC$,
\item $CN^{*}N = NN^{*}C$,
\item $CN(CN)^{\sharp} = (CN)^{\sharp}(CN)$,
\item $||NCh|| = ||N^{*}h||$,
\item $||N^{*}Ch|| = ||Nh||$,
\item $N_+\overset{df}=\tfrac 12(CN+N^*C)$ and $N_{-}\overset{df}=\tfrac 12(CN-N^*C)$ commute,
\item $N^+\overset{df}=\tfrac 12(NC+CN^*)$ and $N^{-}\overset{df}=\tfrac 12(NC-CN^*)$ commute,.
\end{enumerate}
\end{theorem}

\begin{proof} We prove, for instance,  equivalences (1) and (5), (1) and (6). Let's assume (1). From  \eqref{astar} and  \eqref{dfcnor} we have following:
$$NCCN^{*} = CN^{*}NC,$$ and from Def. \ref{def1}
$$NN^{*} = CN^{*}NC.$$
Then, by covering the above equation from the left side by $C$, we get a condition (5). Furthermore, by covering the above equation from the right side by $C$ we get a condition (6).
\end{proof}

\begin{lemma}
Let ${C}$ be a conjugation in $\h$. If $N\in L(\h)$ is  ${C}$-normal then $N_L=CNCN$ and $N_R=NCNC$ are normal.
\end{lemma}
\begin{example}The reverse implication is not true, which follows from the following example. Let $\h=\mathbb{C}^3$, $C(z_1,z_2,z_3)=(\bar z_3,\bar z_2, \bar z_1)$ and $$ N =
\begin{bmatrix}
0&1&0\\
0&0&0\\
0&0&0
\end{bmatrix}.
$$
\end{example}

Next, we will present some results on relations between $C$--normal operators and unitary ones. 

\begin{proposition}
Let ${C}$ be a conjugation in $\h$ and $U\in L(\h)$ be a unitary ope\-ra\-tor, then:
\begin{enumerate}
\item $U$ is $C$--normal,
  \item $U C U^*$ is a conjugation,
  \item $CUC$ is unitary,
  \item if $T$ is ${C}$-symmetric then $UTU^*$ is $UC U^*$-symmetric,
  \item if $N$ is ${C}$-normal then $UNU^*$ is $UC U^*$-normal,
  \item moreover if $U$ is ${C}$-symmetric then
  \begin{enumerate}
\item if $T$ is ${C}$-symmetric then $UTU^*$ is ${C}$-symmetric,
 \item if $N$ is ${C}$-normal then $UNU^*$ is ${C}$-normal.
  \end{enumerate}
\end{enumerate}

\end{proposition}
\begin{proposition}Let $C$ be a conjugation in $\h$ and let $U\in L(\h)$ be unitary operator. An operator $N$ is $C$--normal if and only if  $U^*NCUC$ ($U^*C\,N\,UC$, respectively) is $C$--normal.
\end{proposition}
It is a consequence of the following
\begin{lemma}\label{lem0.2}
Let $X\in LA(\h)$ and let $U\in L(\h)$ be unitary operator. If  $X$ is { antilinearly normal} then
 $U^*XU$ is also  antilinearly normal.
\end{lemma}
\begin{proof}The direct computation shows that
\begin{multline}
(U^*XU)\, (U^*XU)^{\sharp}=U^*XU U^*(U^*X)^{\sharp}=U^*XX^{\sharp}U\\=U^*X^{\sharp}XU=U^*X^{\sharp}U\,U^*X^{\sharp}U=(U^*XU)^{\sharp}(U^*XU).
\end{multline}
\end{proof}

Let $h,g\in \h$ then, by $h\otimes g\in L(\h)$ we will denote rank one operator given by $(h\otimes g)x= \langle x,g\rangle\, h$ for $x\in \h$.

\begin{lemma}\label{lem22}
Let $C$ be a conjugation in $\h$. Let $x,y,h,g\in \h$. Then
\begin{enumerate}
  \item $(h\otimes g)^*=g\otimes h$,
  \item $C(h\otimes g)C=Ch\otimes Cg$,
  \item $(h\otimes g)(x\otimes y)=\langle x,g\rangle\, h\otimes y$.
\end{enumerate}
\end{lemma}

Let $\h$ be a complex Hilbert space with conjugation $C$. Direct calculations show that all $C$--normal rank--one operators have the form $h\otimes Ch$, where $h\in \h$. This operators are $C$--cymmetric, see \cite{KP}. Hence there can be found interesting examples among rank--two or rank--three operators. Let $\dim \h \geqslant 3$. Then, by \cite[Lemma 2.1]{GPP}, there is an orthonormal basis $\{e_k\}$ such that $Ce_k=e_k$. Denote $h=\frac 1{\sqrt 2}(e_1+ie_2)$, $g=e_3$ then $h,Ch,g$ are orthonormal.
Let us consider two operators
\begin{align}A_1&=h \otimes h + h \otimes Ch+Ch\otimes h-Ch \otimes Ch,\\
A_2&=h \otimes Ch+g\otimes h+2\,g\otimes g+2\,Ch\otimes h-Ch \otimes g.
\end{align}
 A direct calculation, using Lemma \ref{lem22}, shows that operators $A_1$ and  $A_2$ are neither $C$--symmetric, nor $C$--skew--symmetric,  but they are $C$-normal.
 Moreover, the operator
 $A_2$ is neither  selfadjoint nor normal.

\section {Finite dimensional case }

Let $ \mathbf{M}_n$ denote the algebra of all $n\times n$ complex matrices. Except the algebra structure, which was recalled, there are some operations on matrices which are defined as follows; let    $M~=~[a_{jk}]\in \mathbf{M}_n$, then we denote
$$\overline M~=~[\bar a_{jk}],\quad M^t~=~[a_{kj}],\quad M^*~=~[\bar a_{kj}],\quad M^s~=~[a_{n-j+1\;\; n-k+1}].$$
We will call the matrix {\it unitary} if its columns (or rows) form an orthonormal basis.

 Let us recall relations between antilinear operators and matrices.
Let $X\in LA(\mathbb{C}^{n})$. Let $e_{1}, \hdots , e_{n} $ be an orthonormal basic in $ \mathbb{C}^{n} $. There is a matrix $M_{X}~=~[a_{jk}]$  
 such that for any $x=\sum_{k=1}^{n}  \langle x, e_{k} \rangle e_{k}\in \mathbb{C}^n$ we have
$$
Xx = \sum\limits_{j=1}^{n}\big(\sum\limits_{k=1}^{n}a_{jk}\ \overline{\langle x, e_{k} \rangle} \big) e_{j}.
$$
Moreover, $a_{jk} = \langle Xe_{k},e_{j} \rangle $.
The matrix $M_X$ will be called a {\it matrix representation} of antilinear operator $X$ as to basis $e_{1}, \hdots , e_{n}$. (The standard matrix for linear operator $T\in L(\mathbb{C}^{n})$ is also denoted by $M_T$.) The following properties hold.
%
%
%
%

\begin{lemma}\label{cacsprzezone}
 Let $X,Y \in LA(\mathbb{C}^{n})$ and $T\in {L}(\mathbb{C}^{n})$. Let $M_{X}, M_{Y}, M_{T} $ be a matrix representation of operators $X$, $Y$, $T$ as to certain orthonormal basis $e_{1}, \hdots , e_{n}$, respectively. Then
\begin{enumerate}
\item[(1)] $M_{XT}=M_{X} \  \overline{M}_{T}$
\item[(2)] $M_{TX} = M_{T} \ M_{X}$
\item[(3)] $M_{XY} = M_{X} \ \overline{M}_{Y}$
\item[(4)]
 $M_{X^{\sharp }} = M_{X}^t$ .
\end{enumerate}
\end{lemma}

There is quite large literature concerning conjugate normal matrices.
\begin{df}[\cite{FI2}]
Matrix $M\in \mathbf{M}_{n}(\mathbb{C})$ is \textit{conjugate-normal} if
$$
MM^{*} = \overline{M^{*}M}.
$$
\end{df}
The theorem bellow shows the relationships between antilinearly normal operators and conjugate normal matrices.

\begin{theorem} \label{thm_1.3}Let $X\in {LA}(\mathbb{C}^{n})$. Then $X$ is antilineary normal if and only if the matrix $ M_X$ is conjugate-normal.
\end{theorem}

\begin{proof} The antilinear operator $X$ is antilinearly normal, if \eqref{def3} is fulfilled,
which is equivalent to
$$
M_{X\,X^{\sharp }} = M_{X^{\sharp }X}.$$
By Lemma \ref{cacsprzezone} we have
$$
M_{X}\overline{M}_{X^{\sharp }} = M_{X^{\sharp }}\overline{M}_{X}
$$
and
$$
M_{X}M^{*}_{X} = \overline{M^{*}}_{X}\overline{M}_{X}
.
$$
\end{proof}
\begin{remark}Let $M\in\mathbf{M}_n$ be a conjugate normal matrix and $M_u$ be an unitary matrix. As it was observed in  \cite[Condition 4.13]{FI2}, the matrix $M_uMM_u^t$ was also conjugate normal. On the other hand, having fixed orthonormal basis, if matrix $M$ is the matrix of some antilinear operator $X\in LA(\mathbb{C}^n)$ i.e. $M=M_X$ and matrix $M_u$ is a matrix of unitary operator $U\in L(\mathbb{C}^n)$ i.e. $M_u=M_U$ then, by Lemma \ref{cacsprzezone}, $M_uMM_u^t=M_UM_XM_U^t=M_{UXU^*}$ and $UXU^*$ is antilineary normal (see Theorem \ref{thm_1.3}) or else Lemma \ref{lem0.2}).
\end{remark}
Recall after \cite{FI1, FI2}
the following theorem characterizing conjugate normal matrices.
\begin{theorem}\label{model1}Let matrix $M\in \mathbf{M}_n$ be conjugate normal. Then there is unitary matrix $M_u\in \mathbf{M}_n$ such that
matrix $M_d=M_uMM_u^t$, where $M_d$ is block diagonal matrix with block diagonal matrices ${(M_d)}^\prime_i$ of size $1\times 1$ and
 ${(M_d)}^{\prime\prime}_j$ of size $2\times 2$ of a form \[ {(M_d)}^\prime_i=[r_i],\quad r_i\geqslant 0\quad\text{ and}\quad
{(M_d)}^{\prime\prime}_j=
\begin{bmatrix}
\phantom{-}s_j&t_j\\
-t_j&s_j
\end{bmatrix},\quad s_j\geqslant 0, t_j\in\mathbb{R} .\]
\end{theorem}

The consequence of the above is the following characterization of $C$--normal operators
\begin{theorem}\label{model2} Let $C$ be a conjugation in $\mathbb{C}^n$. Let  $N\in L(\mathbb{C}^n)$ be a $C$--normal operator. Then, there is unitary operator $U\in L(\mathbb{C}^n) $
such that
\begin{enumerate}
\item
 $N=U^*\, (DC)\, (CUC)$, noticing that $U^*, DC, CUC\in L(\mathbb{C}^n)$ \qquad \qquad or
  \item
  $N=(UC)^\sharp\, (DC)\, CU$, noticing that $(UC)^\sharp, CU\in LA(\mathbb{C}^n)$ and $DC\in L(\mathbb{C}^n)$,
  \end{enumerate}
  where $D$ is block diagonal operator given by block diagonal matrices $(M_d)^\prime_i$ of size $1\times 1$ and
 $(M_d)^{\prime\prime}_j$ of size $2\times 2$ of a form \[ (M_d)^\prime_i=[r_i],\quad r_i\geqslant 0\quad\text{ and}\quad
(M_d)^{\prime\prime}_j=
\begin{bmatrix}
\phantom{-}s_j&t_j\\
-t_j&s_j
\end{bmatrix},\quad s_j\geqslant 0, t_j\in\mathbb{R} .\]
\end{theorem}
\begin{proof}
Operator $N$ is $C$--normal thus $NC$ is antilinearly normal. Let us fix some orthonormal  basis in $\mathbb{C}^n$, for example canonical one. Hence, by theorem \ref{thm_1.3}, the matrix $M_{NC}$ of $NC$ is conjugate normal. Now by Theorem \ref{model1} there is a unitary matrix $M_u$ and specific block diagonal matrix $M_d$ described in Theorem \ref{model1} such that $M_d=M_uM_{NC}M_u^t$. Let $D\in LA(\mathbb{C}^n)$ be an antilinear operator represented by matrix $M_d$ and $U\in L(\mathbb{C}^n)$ be the unitary operator represented by the matrix $M_u$. Then, $M_D=M_UM_{NC}M_U^t=M_{UNCU^*}$ by Lemma \ref{cacsprzezone}. Hence $D=UNCU^*$ and we get (1). Condition (2) can be proved similarly starting with $CN$.
\end{proof}
\section{Case of canonical conjugation in $\mathbb{C}^n$}

Let ${C_{z^n}}$ be a {\it canonical} conjugation in $\mathbb{C}^n$ given by $C_{z^n}(z_1,\dots,z_n)=(\bar{z}_n,\bar{z}_{n-1},\dots,\bar{z}_1)$.
 Recall the model spaces defined in Example \ref{ex2.1}. If we consider the inner function $\theta(z)=z^n$ then $\mathbb{C}^n$ can be seen as a model space $\mathbb{C}^n=H^2\ominus z^n H^2$. Moreover, the conjugation $C_{z^n}$ is exactly the conjugation $C_\theta$ with $\theta=z^n$ considered in Example \ref{ex2.1}.

 \begin{lemma}\label{2.1}Let $T\in L(\mathbb{C}^n)$ and $M_T=[a_{ij}]_{\substack{i=1,\dots,n\\j=1,\dots,n}} $. Then $ M_{C_{z^n}\, T\, C_{z^n}}=[\bar a_{n-i+1\;\; n-j+1}]_{\substack{i=1,\dots,n\\j=1,\dots,n}}$. That means  \[C_{z^n}\bmatrix a_{11} & \cdots &a_{1n}\\ \vdots&\ddots&\vdots\\a_{n1}& \cdots&
a_{nn}\endbmatrix C_{z^n}= \bmatrix a_{nn} & \cdots &a_{n1}\\ \vdots&\ddots&\vdots\\a_{1n}& \cdots& a_{11}\endbmatrix\]
\end{lemma}

By the second diagonal of the matrix $M=[a_{ij}]\in M_{nn}$ we will mean the set of elements $a_{ij}$ such that $i+j=n+1$.
\begin{theorem}
Let $N\in L(\mathbb{C}^n)$ be $C_{z^n}$--normal operator. Then, there is a unitary operator $U \in L(\C^n)$ and the operator $\tilde{D}\in L(\C^n)$ having a matrix representation concentrated on the second diagonal given by block diagonal matrices $(M_d^\prime)_{i}$ of the size
$1 \times 1$ and $(M_d^{\prime \prime})_j$ of the size $2 \times 2$ of the form $(M_d^\prime)_{\iota}=[r_j]$, $r_i\geqslant 0$ and $(M_d^{\prime\prime})_j=\bmatrix t_j & \phantom{-} s_j \\ s_j & -t_j\endbmatrix$, $s_j\geqslant 0$, $t_j\in \mathbb{R}$ such that
\begin{enumerate}
  \item $N=U\, \tilde D\, (C_{z^n}U^*C_{z^n}$), which  can  be written using matrix representation as,
  \item $M_N=M_U\,M_{\tilde D}\,(M_U^{s})^t$.
\end{enumerate}
\end{theorem}
\begin{proof}
By Theorem \ref{model2} (1) there is  a unitary operator $U\in L(\mathbb{C}^n)$ and decomposition $N=U\,(DC_{z^n})\,(C_{z^n}U^*C_{z^n})$ where $DC_{z^n}\in L(\mathbb{C}^n)$ $(C_{z^n}U^*C_{z^n})\in L(\mathbb{C}^n)$. Define $\tilde D=DC_{z^n}\in L(\C^n)$ and applying Lemma \ref{cacsprzezone} the operator $\tilde D$ has got a suitable representation. Hence we get (1). Applying Lemma \ref{2.1} we obtain (2).
\end{proof}

\begin{example}
For $n=3$, having a canonical conjugation $C_{z^3}(z_1,z_2,z_3)$ $=(\bar z_3,\bar z_2, \bar z_1)$, all $C_{z^3}$--normal operators have the matrix representation $M_UM_{\tilde D}(M_U^{s})^t$, where $M_U$ is any unitary matrix and $M_{\tilde D}=\bmatrix 0 & \phantom{-} 0 & r \\ t& \phantom{-}  s & 0 \\ s & -t & 0 \endbmatrix$, $r \geqslant 0$, $s \geqslant 0$ , $t\in \mathbb{R}$ or $M_{\tilde D}=\bmatrix 0 & 0 & r_1 \\ 0& r_2& 0 \\ r_3 & 0 & 0 \endbmatrix$, $r_1,r_2,r_3 \geqslant 0$.
\end{example}

\section{\textbf{$C$-normal operators on  $L^2$  spaces}}
Now, we would like to find examples of $C$--normal operators in $L^2$ spaces. Direct calculation shows the following:
\begin{proposition}Let $(X,\mu)$ be a measure space. Let  $L^{2}(X,\mu)$ be a space of complex valued functions with conjugation $C$ given by $Cf(x)= \overline{f(x)}$.
Let $\varphi \in L^{\infty}$ and $M_{\varphi}$ be a multiplication operator on $L^{2}(X,\mu)$, $M_\varphi f=\varphi f$. Then   $M_{\varphi} $ is  $C$--symmetric, thus also $C$--normal.
\end{proposition}

  Recall that any normal operator $N\in \h$ is unitary equivalent to the multiplication operator $M_\varphi$, i.e. $M_\varphi=UNU^*$ where  $U\in L(H, L^2(X,\mu)$ is unitary. Let $C$ be a conjugation in $H$ such that $(UCU^*)f(x)=\overline {f(x)}$. Then $N$ is $C$--normal. On the other hand, we have the following

\begin{example} Consider $L^{2}[0,1]$. A conjugation $C$ on $L^{2}[0,1]$ is given by $(Cf)(t) = \overline{f(1-t)}$, $t \in [0,1]$.
Let $\varphi \in L^{\infty}$ and consider $M_{\varphi} \in L( L^{2}[0,1])$, $M_\varphi f=\varphi f$. It turns out, that operator $M_{\varphi }$ is C-normal  if and only if $| \varphi |^{2}(t) = | \varphi |^{2} (1-t)$.
\end{example}



\begin{proposition}
Let $ M_{\varphi } \in L^{2}(\mathbb{R}, \tfrac 1{\sqrt{2\pi}}\exp(-\tfrac{x^2}2)\,dx)$ and $\varphi \in L^{\infty }$.  Let conjugation $C$ be given by $Cf(x)=\overline{f(-x)}$. It turns out, that the operator $ M_{\varphi }$ is $C$--normal if and only if $ |\varphi |^{2}$ is an even mapping.
\end{proposition}

%

%

\section{\textbf{$C$--normal Toeplitz operators on  Hardy spaces}}

In the following section, we would like to characterize $C$--symmetric, $C$--skew--symmetric,  $C$--normal operators  in the Hardy space  $H^2$. 

 Recall that $L^2=L^2(\mathbb{T}, m)$ and the  Hardy space $H^2$ is its subspace of those elements of $L^2$ which have negative Fourier coefficient equal to $0$.
Now, we will consider Toeplitz operators. Let $\varphi\in L^\infty=L^\infty(\mathbb{T},m)$ and define the {\it Toeplitz operator with symbol} $\varphi$ as$$ T_\varphi f= P_{H^2}(\varphi f).$$

Note also after \cite[Theorem 9]{BH} that conditions for a Toeplitz operator to  be selfadjoint (i.e. a symbol have to be real) or to be normal (i.e. a symbol have to be linear function of a real function) are very restrictive. In the following section, we will show that, the classes of  $C$--symmetric, $C$--skew--symmetric,  $C$--normal operators  Toeplitz operators  are much more wider. In fact, we fully  characterize these classes of operators with respect to some natural conjugations.

 First natural conjugation (see \cite[p.103]{NF}) which can be studied is given by
\begin{equation}\label{cnat}
(C_0f)(z)=\overline{f(\bar z)}\quad \text{for} \quad f\in H^2.
\end{equation}
In \cite{LK}, for a given real $\xi, \theta$, there was also considered more general  conjugation given by
\begin{equation}\label{cext}
 (C_{\xi , \theta }f)(z) = e^{i\xi }\cdot \overline{f(e^{i\theta }\bar{z})}.
\end{equation} The Hardy space has  the natural basis $e_k(z)=z^k$, $k=0,1,\dots$. Note that $C_{\xi ,\theta }e_{k} = e^{ i\xi }\cdot e^{-ik\theta } e_{k}$, $k\in \mathbb{Z}_+$.
\begin{lemma}\label{lemma2.1}
Let $C_{\xi , \theta }$, $ \xi , \theta \in \mathbb{R}$, be a conjugation on $H^{2}$ given by $ (C_{\xi , \theta }f)(z) = e^{i\xi }\cdot \overline{f(e^{i\theta }\bar{z})}$. Let an operator $T\in L(H^{2})$ be given by a matrix $[a_{lk}]_{k,l \geqslant 0}$ as to the basis $\{e_{k}\}_{k\in \mathbb{Z}_+} $,i.e $a_{lk}=\langle Te_k,e_l\rangle$. Then
\begin{enumerate}
\item the operator $C_{\xi ,\theta } TC_{\xi ,\theta }$ has a matrix $[b_{lk}]_{k,l\geqslant 0}$, $b_{lk}=e^{i (k-l ) \theta }\, \bar{a}_{lk}$,
\item the operator $T$ is $C_{\xi ,\theta }$--symmetric if and only if $a_{lk} = e^{i(k-l)\theta } a_{kl}$, $k,l\geqslant 0$; in particular $a_{ll}$ are arbitrary;
\item  the operator $T$ is $C_{\xi ,\theta }$--skew--symmetric if and only if  $a_{lk} = -e^{i(k-l)\theta } a_{kl}$, $k, l\in \mathbb{Z}_+$; in particular $a_{ll}=0$.
\end{enumerate}
\end{lemma}

\begin{proof}
 To see (1) let us compute
\begin{multline*}
b_{lk}=\langle C_{ \xi ,\theta } TC_{\xi ,\theta } e_{k} , e_{l} \rangle  = 
{ \langle  C_{ \xi ,\theta } e_{l} , TC_{ \xi ,\theta } e_{k} \rangle  }=
\overline{ \langle TC_{ \xi ,\theta } e_{k} , C_{ \xi ,\theta } e_{l}   \rangle  }  \\ 
= \overline{ \langle Te^{i\xi } e^{-ik\theta }e_{k} , e^{i\xi }e^{-il\theta }e_{l}    \rangle } = e^{i(k-l)\theta }\, \overline{ \langle Te_{k} , e_{l} \rangle } = e^{i(k-l)\theta }\,\bar{a}_{lk}.
\end{multline*}
Conditions (2) and (3) follows from (1) and appropriate definitions.
\end{proof}

\begin{col}
Let $C_0$ be a conjugation on $H^2$ given by $(C_0f)z = \overline { f(\bar{z})}$, \, $f\in H^2$. Let $T\in L(H^{2})$ be given by the matrix $[a_{kl}]_{k,l\geqslant 0}$ according to the basis $\{e_{k}\}_{k\in \mathbb{Z}_+}$. Then, $T$ is $C_0$-symmetric if and only if $a_{kl} = a_{lk}$, $k,l=0, 1, 2 , \hdots  $, and $T$ is $C_0$--skew--symmetric if and only if $a_{ll}=0$, $a_{kl} = - \,a_{lk}$, $k,l=0, 1, 2 , \hdots  $,
\end{col}

\begin{proposition}\label{lemma2}
Let $\varphi \in L^{\infty }$ have a Fourier expansion  $ \varphi (z) = \sum \limits ^{+ \infty } _{-\infty } \widehat{\varphi} (n)z^{n}$. The Toeplitz operator $T_{\varphi }$ has the matrix $ [a_{lk}]_{k,l = 0,1,2, \hdots  }  $ and $a_{lk} = \widehat{\varphi} (l-k)$. Then
\begin{enumerate}
\item the operator $C_{\xi ,\theta } T_\varphi C_{\xi ,\theta } $ has matrix $[b_{lk}]$ with $b_{lk}\! = e^{i(k-l)\theta }\overline{\widehat{\varphi} (l-k)}$
\item the Toeplitz operator $T_{\varphi }$ is $C_{\xi ,\theta }$--symmetric if and only if $\widehat{\varphi} (-k)$ $ = e^{ik\theta }\widehat{\varphi} (k)$, $k\in \mathbb{Z}$; in particular $\widehat{\varphi} (0)$ is arbitrary;
\item the operator $T_{\varphi }$ is $C_{\xi ,\theta }$--skew--symmetric if and only if $\widehat{\varphi} (-k) = -e^{ik\theta }\widehat{\varphi} (k)$, $k\in \mathbb{Z}$; in particular $\widehat{\varphi} (0)=0$ if $\Arg\theta \not=\pi$ and  $\widehat{\varphi} (0)$ is arbitrary if $Arg\, \theta =\pi$.
\end{enumerate}
\end{proposition}

\begin{proposition}\label{prop2}
Let $C_{\xi ,\theta }$, $\xi ,\theta \in \mathbb{R} $, be a conjugation on $H^{2}$ given by $ (C_{\xi ,\theta }f)(z) = e^{i\xi }\overline{f(e^{i\theta }\bar{z})} $. Let $\varphi \in L^{\infty }$, $\varphi (z) = \sum ^{+ \infty }\limits _{n = -\infty }  \widehat{\varphi }(n)z^{n}$ and denote $\varphi_+ (z) = \sum ^{+ \infty }\limits _{n = 1 }  \widehat{\varphi }(n)z^{n}$,
$\varphi_{-} (z) = \sum ^{-1 }\limits _{n = -\infty }  \widehat{\varphi }(n)z^{n}$. If $T_{\varphi }$ is $C_{\xi ,\theta }$--normal then  there is $\eta $, $|\eta |= 1$ such that \begin{equation}\widehat{\varphi }(-k) = \eta e^{ik\theta }\widehat{\varphi}(k)\quad \text{for}\ \ k=1,2, \hdots
\end{equation}
or equivalently there is $\eta $, $|\eta |= 1$ such that
\begin{equation}\label{warmin}
\varphi_{-}=\eta\ e^{i\xi}\ \overline{C_{\xi,\theta}\varphi_{+}}
\end{equation}
\end{proposition}


\begin{remark}\label{remark1} Let us consider $\varphi,\psi \in L^{\infty }$ with the Fourier expansion  $\varphi (z) = \sum ^{+ \infty }\limits _{n = -\infty }  \widehat{\varphi }(n)z^{n}$ and $\psi (z) = \sum ^{+ \infty }\limits _{n = -\infty }  \widehat{\psi }(n)z^{n}$, respectively.
Let $T_{\varphi },T_{\psi }$ be Toeplitz operators on $H^{2}$.
The operator  $T_{\varphi }T_{\psi }$ is  not always a Toeplitz operator. In fact, as it was shown in  \cite{BH} that
 \begin{equation}\label{comfor} \langle T_{\varphi } T_{\psi } e_{k+1},e_{l+1} \rangle  - \langle T_{\varphi } T_{\psi } e_{k},e_{l}  \rangle =\widehat\varphi (l+1)\ \widehat\psi(-k-1).\end{equation}
\end{remark}

\begin{proof}[Proof of Proposition~\ref{prop2}]
 Applying Remark \ref{remark1} we have
\begin{multline}\label{wz1}
\langle (S^{*}T_{\bar{\varphi }}T_{\varphi } S - T_{\bar{\varphi }}T_{\varphi })e_{k}  ,e_{l}  \rangle  =
\langle  T_{\bar{\varphi }}T_{\varphi } Se_{k} , Se_{l}  \rangle  - \langle   T_{\bar{\varphi }}T_{\varphi } e_{k},e_{l} \rangle
\\
 =\langle  T_{\bar{\varphi }}T_{\varphi } e_{k+1} , e_{l+1}  \rangle  - \langle   T_{\bar{\varphi }}T_{\varphi } e_{k},e_{l} \rangle = \overline{\widehat\varphi({-l-1})}\, \,\widehat\varphi(-k-1).
\end{multline}
On the other hand, also using Lemma \ref{lemma2.1} and Remark \ref{remark1}, we get
\begin{align*}
 \langle (S^{*}C_{\xi ,\theta }T_{\varphi }T_{\bar{\varphi }}& C_{\xi ,\theta }S - C_{\xi ,\theta }T_{\varphi }T_{\bar{\varphi }} C_{\xi ,\theta })e_{k},e_{l}   \rangle   \\
&=\langle C_{\xi ,\theta }T_{\varphi }T_{\bar{\varphi }} C_{\xi ,\theta } e_{k+1}, e_{l+1}   \rangle - \langle C_{\xi ,\theta }T_{\varphi }T_{\bar{\varphi }} C_{\xi ,\theta } e_{k},e_{l}   \rangle  \\
&=e^{i(k-l)\theta }\  \overline{ \langle T_{\varphi }T_{\bar{\varphi }} e_{k+1},e_{l+1}   \rangle } - e^{i(k-l)\theta }\  \overline{ \langle  T_{\varphi }T_{\bar{\varphi }} e_{k},e_{l}  \rangle  }\\ &=  e^{i(k-l)\theta }\
\overline{\widehat\varphi({l+1})}\,{\widehat\varphi(k+1)}.
\end{align*}
The last equality follows from \eqref{wz1} for $T_{\varphi }T_{\bar{\varphi }} $. 
If $T_{\varphi }$ is $C_{\xi ,\theta } $- normal, by Theorem \ref{twCnor} (5), subtracting both  sides we get
\begin{align}\label{eq21}
e^{i(k-l)\theta }\
\overline{\widehat\varphi({l+1})}\,{\widehat\varphi(k+1)}=\overline{\widehat\varphi({-l-1})}\, \,\widehat\varphi(-k-1) 
\end{align}
for $ k,l = 0,1,2,\hdots$.
Assume for the while that $ \widehat{\varphi } (k) \neq 0$, $k=\pm 1, \pm 2, \hdots $ Thus
\begin{equation}\label{wz31}
\overline{\bigg(\frac{\widehat{\varphi}(-l)}{e^{il\theta }\ \widehat{\varphi} (l)}\bigg)} = \bigg( \frac{\widehat{\varphi } (-k)}{e^{ik\varphi }\ \widehat{\varphi}  (k)}\bigg)^{-1}
\end{equation}
for $k,l=1,2, \hdots  $
Hence, there is $\eta$ such that $  \frac{\widehat{\varphi } (-k)}{e^{ik\varphi }\widehat{\varphi } (k)} = \eta$ for $k=1,2,\hdots $. Moreover, by \eqref{wz31}, we get $|\eta | = 1 $.
Thus 
\begin{equation}\label{eq22}
\widehat{\varphi } (-k)= \eta\, e^{ik\theta }\,\widehat{\varphi } (k)\quad  \text{for}\quad k=1,2\hdots.
\end{equation}
If $\widehat{\varphi } (k)=0$ and \eqref{eq22} is fulfilled then $\widehat{\varphi } (-k)=0$ and \eqref{eq21} holds.
\end{proof}


\begin{theorem}\label{mainch2}
Let $C_{\xi ,\theta }$, $\xi ,\theta \in \mathbb{R} $, be a conjugation on $H^{2}$ given by $ (C_{\xi ,\theta }f)(z) = e^{i\xi }\overline{f(e^{i\theta }\bar{z})} $. Let $\varphi \in L^{\infty }$, $\varphi (z) = \sum ^{+ \infty }\limits _{n = -\infty }  \widehat{\varphi }(n)z^{n}$ and denote $\varphi_+ (z) = \sum ^{+ \infty }\limits _{n = 1 }  \widehat{\varphi }(n)z^{n}$,
$\varphi_{-} (z) = \sum ^{-1 }\limits _{n = -\infty }  \widehat{\varphi }(n)z^{n}$. Then $T_{\varphi }$ is $C_{\xi ,\theta }$-normal if and only if there is $\eta $, $|\eta |= 1$ such that
\begin{align}
\varphi_{-}=\eta\ e^{i\xi} \ \overline{C_{\xi,\theta}\varphi_{+}}&\qquad \text{and}\label{warunek1}\\
(\eta-\bar\eta) \varphi_{+}\,C_{\xi,\theta}\varphi_{+}+\overline{\widehat{\varphi}(0)} (\eta-1)&e^{i\xi}\varphi_{+}-\widehat{\varphi}(0)(\bar\eta-1)C_{\xi,\theta}\varphi_{+}=0.\label{warunek2}
\end{align}
\end{theorem}

Denote by  $\varphi_\sim^{\theta}(z)=\,e^{-i\xi}\,C_{\xi,\theta}\varphi_{+}(z)=\,\overline{\varphi_+(e^{i\theta}\bar z)}$.
Easy to see that $\overline{\varphi_\sim^{\theta}}=\bar{\varphi}_\sim^{\theta}$.

\begin{lemma} \label{fiteta}
With the notation above the following holds:
\begin{enumerate}
  \item $C_{\xi ,\theta}T_{\varphi _{+}} C_{\xi ,\theta}=T_{\varphi_\sim ^{\theta}}$,
  \item $C_{\xi ,\theta}T_{\bar{\varphi}_{+}} C_{\xi ,\theta}=T_{\bar{\varphi}_\sim ^{\theta}}$,
  \item $C_{\xi ,\theta}T_{\varphi_\sim ^{\theta}} C_{\xi ,\theta}=T_{\varphi _{+}}$,
  \item $C_{\xi ,\theta}T_{\bar{\varphi}_\sim ^{\theta}} C_{\xi ,\theta }=T_{\bar{\varphi} _{+}}$.
\end{enumerate}
\end{lemma}

\begin{proof}
To see (1) let us calculate for $f,g\in H^2$ :
\begin{align*}
\langle C_{\xi ,\theta } & T_{\varphi _{+}} C_{\xi, \theta} f,g\rangle =  \langle C_{\xi, \theta} g, T_{\varphi _{+}} C_{\xi , \theta} f \rangle =  \langle C_{\xi , \theta}g, P_{H^2}M_{\varphi _{+}}C_{\xi , \theta}f \rangle = \\ &= \langle C_{\xi , \theta}g,M_{\varphi _{+}}C_{\xi ,\theta}f\rangle =\int e^{i\xi\,}\overline{g(e^{i\theta}\bar z)}\ \overline{\varphi _{+}(z)\, e^{i\xi}\,\overline{f(e^{i\theta}\bar z)}}\,dm(z) = \\ &= \int\bar{\varphi}_{+}(z) f(e^{i\theta}\bar z)\,\overline{g(e^{i\theta}\bar z)}\,dm(z).
\end{align*}
Let us substitute $\omega=e^{i\theta}\bar z$. Then $z=e^{i\theta}\bar{\omega}$. Thus
\begin{align*}
\langle C_{\xi ,\theta}T_{\varphi _{+}} C_{\xi ,\theta} f,g\rangle = \int \overline{\varphi_{+}(e^{i\theta}\bar{\omega})}\ f(\omega)\,\overline{g(\omega)}\,dm(\omega)= \langle T_{\varphi_{\sim}^{\theta}} f,g\rangle .
\end{align*}
The property (3) follows from (1) since $(\varphi_\sim^{\theta})_\sim^{\theta}=\varphi$ and (2), (4) follows from (1) and (3) taking $\bar{\varphi}$ instead of $\varphi$.
\end{proof}

\begin{proof}[Proof of Theorem~\ref{mainch2}]
 Let us apply Proposition \ref{prop2} and by \eqref{warmin} operator $T_\varphi$ being $C_{\xi ,\theta }$-normal has to be represented as
$$T_\varphi=T_{\varphi_{+}}+\widehat{\varphi}(0) I+\eta e^{i\xi}\ T_{\overline{C_{\xi,\theta}\varphi_{+}}}=
T_{\varphi_{+}}+\widehat{\varphi}(0) I+\eta T_{\bar\varphi_{\sim}^\theta}.
$$
 Therefore $$T_{\varphi }^*=T_{\bar{\varphi }_{+}} + \overline{ \widehat{\varphi}(0)} I+\bar \eta  T_{\varphi_\sim ^{\theta}}.$$ Let  us calculate:
\begin{multline*}
T_{\varphi}T_{\varphi}^{*}=T_{\varphi_{+}}T_{\bar{\varphi}_{+}}+\overline{\widehat{\varphi}
(0)}\, T_{\varphi_{+}} +\bar \eta\,  T_{\varphi_{+}} T_{\varphi_{\sim}^{\theta}}+\\ +\widehat{\varphi}(0)\, T_{\bar\varphi_{+}}+|\widehat{\varphi}(0)|^2\ I +\widehat{\varphi}(0)\, \bar \eta\,  T_{\varphi_{\sim}^{\theta}}+\\ +\eta  T_{\bar{\varphi}_{\sim}^{\theta}} T_{\bar{\varphi}_{+}} +
\overline{\widehat{\varphi}(0)}\, \eta\,  T_{\bar{\varphi}_{\sim}^{\theta}} + |\eta|^2\, T_{\bar{\varphi}_{\sim}^{\theta}} T_{\varphi_{\sim}^{\theta}}.
\end{multline*}
Hence, by Lemma \ref{fiteta} we will get
\begin{multline*}
C_{\xi ,\theta} T_{\varphi}T_{\varphi}^{*} C_{\xi ,\theta}=T_{\varphi_{\sim}^{\theta}} T_{\bar{\varphi}_{\sim}^{\theta}} + \widehat{\varphi}(0)\, T_{\varphi_{\sim}^{\theta}} +\eta\,  T_{\varphi_{\sim}^{\theta}}T_{\varphi_{+}}+\\ +\overline{\widehat{\varphi}(0)}\,  T_{\bar{\varphi}_{\sim}^{\theta}}+ |\widehat{\varphi}(0)|^2\ I +\overline{\widehat{\varphi}(0)} \eta \, T_{\varphi_{+}} +\\+  \bar\eta \, T_{\bar{\varphi}_{+}} T_{\bar{\varphi}_{\sim}^{\theta}}+ \widehat{\varphi}(0)\bar\eta \,T_{\bar\varphi_{+}}+T_{\bar\varphi_{+}}T_{\varphi_{+}}.
\end{multline*}
On the other hand, we have
\begin{multline*}
T_{\varphi}^{*}T_{\varphi}=T_{\bar\varphi_{+}}T_{\varphi_{+}} +\widehat{\varphi}(0)\, T_{\bar{\varphi}_{+}}+\eta\, T_{\bar{\varphi}_{+}} T_{\bar{\varphi}_\sim ^{\theta}}+ \\ +
\overline{\widehat{\varphi}(0)}\, T_{\varphi_{+}} + |\widehat{\varphi}(0)|^2\, I+ \eta\overline{\widehat{\varphi}(0)}\, T_{\bar{\varphi}_\sim^{\theta}}+\\ +\bar\eta\, T_{\varphi_\sim^{\theta}} T_{\varphi _{+}} + \bar\eta \widehat{\varphi}(0)\, T_{\varphi_\sim ^{\theta}}+T_{\varphi_\sim ^{\theta}} T_{\bar{\varphi}_\sim^{\theta}}.
\end{multline*}
Since $\varphi_+$ is analytic and $\bar\varphi_+ $ is coanalytic thus by \cite{BH}, we have following
\begin{equation}\label{opzero}
\begin{alignedat}{23}
C_{\xi \theta} T_{\varphi}T_{\varphi}^{*} C_{\xi \theta} - T_{\varphi}^{*}T_{\varphi}  &=  (\eta  - \bar\eta ) T_{\varphi _{\sim}^{\theta } \varphi _{+}} + (\bar\eta  - \eta ) T_{\bar\varphi _{+}\bar{\varphi}_\sim^{\theta}}  + \\ &+ (\widehat{\varphi}(0) - \bar\eta  \widehat{\varphi}(0)) T_{\varphi_\sim^{\theta}} + (\overline{\widehat{\varphi}(0)} - \eta  \overline{\widehat{\varphi}(0)})T_{\bar\varphi _{\sim } ^{\theta }} + \\&+ (\overline{\widehat{\varphi}(0)} \eta  - \overline{\widehat{\varphi}(0))}T_{\varphi _{+}} + (\widehat{\varphi}(0)\bar\eta  - {\widehat{\varphi}(0))}T_{\bar\varphi _{+}}.
\end{alignedat}
\end{equation}
 The condition for the operator $T_\varphi$ to be $C_{\xi ,\theta }$--normal is that the operator above has to be zero. In fact the operator above is  a Toeplitz one with the symbol (let say) $\psi\in L^\infty\subset L^2$. Thus the symbol $\psi$ has to be a zero. Hence, the analytic and co-analytic part, which are complex adjoint one to the other, of $\psi$ have to be 0.
Extracting the analytical part of the function $\psi$ we get:
\begin{align*}0=&
(\eta  - \bar\eta ) {\varphi _{+}\varphi _{\sim}^{\theta }} +
 \overline{\widehat{\varphi}(0)}\, (  \eta-1){\varphi _{+}} +
 \widehat{\varphi}(0) ( 1-\bar\eta) {\varphi_\sim^{\theta}}\\=&
(\eta-\bar\eta)e^{-i\xi} \varphi_{+}\,C_{\xi,\theta}\varphi_{+}+\overline{\widehat{\varphi}(0)} (\eta-1)\varphi_{+}-\widehat{\varphi}(0)(\bar\eta-1)e^{-i\xi}C_{\xi,\theta}\varphi_{+}.
\end{align*}
Hence we get \eqref{warunek2}.

Arguing the other direction, if \eqref{warunek1} and \eqref{warunek2} are fulfilled
 the operator considered in \eqref{opzero} have to be zero. \end{proof}


\begin{example}If, in Theorem \ref{mainch2}, the existing $\eta$ is real, then we have the following cases:
\begin{enumerate}
\item Let $\eta = 1 $ then \eqref{warunek2} is fulfilled and \eqref{warunek1} means that operator $T_{\varphi }$ is $C_{\xi , \theta }$-symmetric, see Lemma \ref{lemma2}, (2).
\item Let $\eta = -1 $ and $\widehat{\varphi}(0) = 0 $ then \eqref{warunek2} is fulfilled and \eqref{warunek1} with $\widehat{\varphi}(0) = 0 $ means that operator $T_{\varphi }$ is $C_{\xi , \theta }$-skew--symmetric, see Lemma \ref{lemma2}, (3).
\item For $\eta = -1 $, $\widehat{\varphi}(0)\neq  0$, $\Arg\theta\not=\pi$,  condition \eqref{warunek2} is equivalent to
\begin{equation}
\overline{\widehat{\varphi}(0)}\,\varphi_+=\widehat{\varphi}(0)\,e^{-i\xi}\,C_{\xi,\theta}\varphi_{+}=\widehat{\varphi}(0)
\varphi_\sim^{\theta}.\end{equation}
Hence, in this case, the operator $T_{\varphi }$ is $C_{\xi , \theta }$-normal ( but neither $C_{\xi , \theta }$-symmetric nor  $C_{\xi , \theta }$-skew-symmetric) for $\varphi\in L^\infty$ if \begin{align*}&\hat\varphi(-k)=-e^{ik\theta}\hat\varphi(k) \quad \text{for }\ k=1,2,\dots\quad\text{and}\\
&\Arg\hat\varphi(k)\overset{\text{mod}\,2\pi}{=}\Arg\hat\varphi(0)-\tfrac {k}{2}\theta  \quad \text{for }\ k=1,2,\dots
\end{align*}
\end{enumerate}
\end{example}


It is worth to notice the special case of  Theorem  \ref{mainch2}.
\begin{col}
Let $C_{0 }$, be a conjugation on $H^{2}$ given by {$ (C_{0}f)(z) = \overline{f(\bar{z})} $} for $f\in H^2$. Let $\varphi \in L^{\infty }$ and {$\varphi = \varphi_{-} + \widehat{\varphi}(0) + \varphi _{+ }  $}. Then, the Toeplitz operator $T_{\varphi }$ is $C_{0 }$-normal if and only if {there is $\eta $,} $|\eta |= 1$ such that
\begin{enumerate}
	\item $ \varphi _{-} = \eta \overline{ C_{0} \varphi _{+}} $ and
	\item $ (\eta-\bar\eta) \varphi_{+}\,C_{0}\varphi_{+}+\overline{\widehat{\varphi}(0)} (\eta-1)\varphi_{+}-\widehat{\varphi}(0)(\bar\eta-1)C_{0}\varphi_{+}=0 $
	\end{enumerate}
\end{col}

\begin{example} 
 Let $s \in (-1;1)$ and let  $\varphi(z)=\frac{ - s \bar{z} }{1-i s \bar{z}}+{(\tfrac 12+\tfrac 12 i)} + \frac{i s z}{1- i s  z}$. Conditions (1) and (2) of Corollary are fulfilled for $\eta = i$.  Thus $T_{\varphi}$ is $C_0$--normal but neither $C_0$-symmetric nor $C_0$-skew-symmetric by
 Lemma \ref{lemma2}.
\end{example}

\section{Composition Operators}

Let $(X,\Sigma,\mu)$ be a measure space with a non--negative $\sigma$-finite measure $\mu$ and consider a space $L^2(X,\Sigma,\mu)$. Then a measurable function $T \colon X \to X$ induces a composition operator $C_Tf=f \circ T$. It is known \cite{RW} that if $C_T$ is bounded then $\mu \circ T^{-1}$ is absoluty continuous with respect to $\mu$ and  the Radon-Nikodym derivative $h=\tfrac {d\mu \circ T^{-1}}{d\mu}$ is essentially bounded. Conversely, if $T$ satisfies this conditions, function $T$ induce bounded linear operator $C_T$ on $L^2(X,\Sigma,\mu)$. It is clear that $h$ is always nonnegative.
Note also the basic formula  \begin{equation}\label{cder}\int C_Tf\,d\mu=\int f\circ T\,d\mu=\int fh\,d\mu.\end{equation}

\begin{proposition}Take the conjugation $C$ in $ L^2(X,\Sigma,\mu)$ given by $C(f)(x)=\overline{f(x)}$.
Assume that $C_T$ is a bounded composition operator given by a measurable function $T \colon X \to X$. Then
following are equivalent:
\begin{enumerate}
  \item $C_T$ is $C$--normal,
  \item $C_T$ is normal.
\end{enumerate}
\end{proposition}

\begin{proof}
To show equivalence of (1) to (2) we will show that $CC_T^*C_TC=C_T^*C_T$. Let $f,g \in L^2(X,\Sigma,\mu)$ then
\begin{multline*}
\langle CC_T^*C_TCf,g\rangle = \langle Cg,C_T^*C_TCf\rangle = \langle C_TCg,C_TCf\rangle\\ = \int (Cg\circ T)\cdot \overline{Cf\circ T}\,d\mu = \int (\bar{g}\circ T)\ ( f \circ T)\,d\mu\\= \int\bar{g}f\,h\,d\mu=\langle C_Tf,C_Tg\rangle = \langle C_T^*C_Tf,g \rangle.\end{multline*}
\end{proof}
 Let us note that $(Cf)x=\overline{f(-x)}$ gives us  a conjugation in $L^2(\mathbb{R},m)$, ($m$ Lebesgue measure). On the other hand, $(Cf)x=\overline{f(1-x)}$ defines a conjugation on the space $L^2([0,1], m)$. Consider the general space $L^2(X,\mu)$, where $(X,\mu)$ is a measure space with non-negative measure $\mu$.
The above two situations lead to the following:


\begin{proposition}\label{lem1}
Let $(X,\Sigma,\mu)$ be a measure space with a non--negative measure $\mu$ and the antilinear operator $C\colon L^2(X,\Sigma,\mu)\to L^2(X,\Sigma,\mu)$  given by $(Cf)(x) = \overline{f(\alpha(x))} $, where $\alpha\colon X\to X$ is measurable. Then, $C$ is conjugation if and only if
\begin{enumerate}
\item $\alpha ^{2} = I_{X}$,
\item $ \mu = \mu\circ\alpha $.
\end{enumerate}
\end{proposition}

\begin{proof}
For $f \in L^2(X,\Sigma,\mu)$ and $x\in X$ we have $$ (C^{2}f)(x) = C(Cf)(x) = \overline{Cf(\alpha (x))} =  f(\alpha ^{2}(x)). $$ Hence $C^2=I$ is equivalent to $\alpha ^{2}=I_X$. For the second condition, for any $f,g\in L^2(X, \Sigma,\mu)$, let us  calculate
\begin{align*}
\langle Cf, Cg \rangle &= \int (Cf)(x)\overline{(Cg)(x)}d\mu (x) = \int \overline{f(\alpha (x))} \cdot g(\alpha (x)) d\mu(x)
\end{align*}
and
\begin{align*}
\langle g, f\rangle = \int g(x)\overline{f(x)} d\mu (x).
\end{align*}
Hence the equality of two above for all $f,g$ gives $\mu=
\mu\circ \alpha^{-1}=\mu\circ\alpha$.
\end{proof}

\begin{theorem}\label{prop11} Let $L^2(X,\Sigma, \mu)$ with conjugation $C$ given by $(Cf)(x) = \overline{f(\alpha(x))} $, i.e. $\alpha \colon X \to X$ be measurable function with  $\alpha ^{2} = I_{X}$
  and $ \mu = \mu\circ\alpha $. Assume that $C_T$ is a bounded composition operator given by a measurable function $T \colon X \to X$.
 Then, the operator $C_T$ is $C$--normal if and only if
\begin{enumerate}
  \item $T^{-1}(\Sigma)$ is essentially all $\Sigma$, i.e. for a given $\omega \in\Sigma$ there is $\tilde\omega\in\Sigma$ such that $m\big((T^{-1}(\tilde\omega)\setminus \omega)\cup (\omega\setminus T^{-1}(\tilde\omega))\big)=0$, and
  \item $h \circ T=h \circ \alpha$ $\mu$ a.e., where $h=\tfrac {d\mu \circ T^{-1}}{d\mu}$.
\end{enumerate}
\end{theorem}

\begin{proof}
For $f,g \in L^2(X,\mu )$ we have
\begin{align*}
\langle CC_T^*C_TCf,g\rangle &= \langle C_TCg,C_TCf \rangle =\int (Cg\circ T)\, \overline{(Cf\circ T)}\,d\mu \\ &= \int (\bar{g}\circ \alpha\circ T)\ (f\circ\alpha \circ T) \, d\mu \\ &= \int (\bar{g}\circ \alpha) \; (f\circ \alpha)\  h\,d\mu = \int f\,\bar{g}\  (h\circ \alpha^{-1})\,d\mu \circ \alpha ^{-1}
\end{align*}
Then, since $\alpha=\alpha^{-1}$, $$CC_T^*C_TCf =(h\circ \alpha^{-1})\cdot f.$$ If $f$ belongs to range of $C_T$ then $f=C_Tf_0$ and
\begin{align*}
C_TC_T^*f &=C_TC_T^*C_Tf_0=C_TCCC_T^*C_TCCf_0\\ &=C_TC\,\big(CC_T^*C_TC\big)\,(Cf_0)=C_TC((h\circ \alpha)\cdot (Cf_0))\\ &=C_T((\bar{h}\circ\alpha\circ\alpha)\,\cdot\,C(Cf_0))=C_T(h\cdot f_0)\\ &=(h\circ T)\cdot (C_Tf_0)=(h\circ T)\cdot f.
\end{align*}
If $C_T$ is $C$--normal then $$(h\circ \alpha)f=(h\circ T)f$$ for all $f$ in range of $C_T$. The rest of the proof is analogous as the proof of \cite[Lemma 2]{RW}.\end{proof}

\begin{example}Let us consider $L^2(\mathbb{R},m)$ with the conjugation $(Cf)x=\overline{f(-x)}$, $\alpha(x)=-x$. Let $T(x)=-x$ for $x\geqslant 0$ and $T(x)=-2x$ for $x<0$.
Then  the Radon--Nikodym derivative $h=\frac{dm \circ T^{-1}}{dm}$ is given by $h(x)=\tfrac 12$ for $x\geqslant 0$ and $h(x)=1$ for $x<0$. It is clear that $h\circ \alpha=h\circ T$, thus $C_T$ is $C$--normal. Furthermore, $h\not=h\circ T$ thus $C_T$ is not normal (see \cite[Lemma 2]{RW}) and direct calculation shows that it is also always neither $C$--symmetric nor $C$--skew--symmetric.
\end{example}%
%


\end{document}